\documentclass[leqno]{amsart}
\usepackage{amsfonts,amssymb,amsmath,amsgen,amsthm}
\usepackage{hyperref}
\usepackage{color}
\usepackage{epsfig,epstopdf,color,bm}
\usepackage{graphicx,wrapfig}
\newcommand{\msc}[2][2000]{%
  \let\@oldtitle\@title%
  \gdef\@title{\@oldtitle\footnotetext{#1 \emph{Mathematics subject
        classification.} #2}}%
}

\theoremstyle{plain}
\newtheorem{theorem}{Theorem} [section]
\newtheorem{definition}[theorem]{Definition}

\newtheorem{lemma}[theorem]{Lemma}

\newtheorem{proposition}[theorem]{Proposition}

\theoremstyle{remark}
\newtheorem{remark}[theorem]{Remark}

\def\C{{\mathbb C}}% complex numbers
\def\R{{\mathbb R}}% real numbers
\def\N{{\mathbb N}}% nonnegative integers
\def\F{\mathcal F}

\def\({\left(}
\def\){\right)}
\def\<{\left\langle}
\def\>{\right\rangle}
\def\le{\leqslant}
\def\ge{\geqslant}
\def\d{{\partial}}
\def\eps{\varepsilon}

\def\Tend#1#2{\mathop{\longrightarrow}\limits_{#1\rightarrow#2}}

\def\d{{\partial}}
\def\eps{\varepsilon}

\newcommand{\be}{\begin{equation}}
\newcommand{\ee}{\end{equation}}

\DeclareMathOperator{\RE}{Re}
\DeclareMathOperator{\IM}{Im}

\numberwithin{equation}{section}

\begin{document}

\title[Gaussons under  repulsive harmonic potential]{Nonuniqueness and
  nonlinear instability of Gaussons under  repulsive harmonic potential}
\author[R. Carles]{R\'emi Carles}
\author[C. Su]{Chunmei Su}

\address{Univ Rennes, CNRS\\ IRMAR - UMR 6625\\ F-35000
  Rennes, France}
\email{Remi.Carles@math.cnrs.fr}

\address{Yau Mathematical Sciences Center\\Tsinghua University\\
  Beijing 100084, China}
\email{sucm@tsinghua.edu.cn}

\begin{abstract}
We consider the Schr\"odinger equation with a nondispersive
logarithmic nonlinearity and a repulsive harmonic potential. For a
suitable range of the coefficients, there exist two positive
stationary solutions, each one generating a continuous family of
solitary waves. These solutions are Gaussian, and turn out to be
orbitally unstable. We also discuss the notion of ground state in this
setting: for any natural definition, the set of ground states is
empty.
\end{abstract}
\thanks{RC is supported by Rennes M\'etropole through its AIS program. }


\maketitle

\section{Introduction}
\label{sec:intro}

We consider the equation
\begin{equation}
  \label{eq:logNLSrep}
  i\d_t u +\frac{1}{2}\Delta u = -\omega^2\frac{|x|^2}{2}u +\lambda
  u\ln\(|u|^2\),\quad x\in \R^d,
\end{equation}
in the case $\omega>0$ (repulsive harmonic potential) and
$\lambda<0$. The logarithmic Schr\"odinger equation
(\eqref{eq:logNLSrep} with $\omega=0$) was introduced in
\cite{BiMy76}, and has been considered in various fields of physics
since; see e.g. \cite{BEC,buljan,hansson,Hef85,KEB00,DMFGL03,Zlo10}
and references therein.  A special feature of the logarithmic
nonlinearity is that it leads to very special solitary waves, called
\emph{Gaussons} in \cite{BiMy76,BiMy79}: if $\lambda<0$, for any $\nu\in \R$,
\begin{equation*}
  e^{i\nu t} e^{\frac{\lambda d-\nu}{2\lambda}}e^{\lambda |x|^2}
\end{equation*}
is a solution to \eqref{eq:logNLSrep} (with $\omega=0$). These
solitary waves are orbitally stable, as proved in \cite{Caz83} (radial
case) and \cite{Ar16} (general case). In addition, still in the case
$\omega=0$, it is known that for $\lambda<0$, no solution is
  dispersive (\cite[Proposition~4.3]{Caz83}), while for $\lambda>0$,
every solution is dispersive, with an enhanced rate compared to the usual rate
of the free Schr\"odinger equation (\cite{CaGa18}).
\smallbreak

The logarithmic Schr\"odinger equation in the presence of a confining
harmonic potential was considered in physics in \cite{Bouharia2015},
\begin{equation}
  \label{eq:logNLSconf}
  i\d_t u +\frac{1}{2}\Delta u = \omega^2\frac{|x|^2}{2}u +\lambda
  u\ln\(|u|^2\),\quad x\in \R^d.
\end{equation}
 In the case $\lambda<0$ (\cite{ACS20}) as well as in the case
 $\lambda>0$ (\cite{CaFe-p}), generalized Gaussons exist, and are
 orbitally stable, in the sense introduced in \cite{CaLi82} (see Definition~\ref{def:stability} below for the definition in
 the case of \eqref{eq:logNLSrep}, the notion being the same for
 \eqref{eq:logNLSconf}).
 \smallbreak

 The case of an inverted, or repulsive harmonic potential as in
 \eqref{eq:logNLSrep}, does not seem to correspond to a realistic
 model, but constitutes an interesting mathematical toy. The potential
 $V(x) =-\omega^2\frac{|x|^2}{2}$
 is unbounded from below, and goes to $-\infty$ as fast as possible in
 order to guarantee that the Hamiltonian $-\frac{1}{2}\Delta +V(x)$ is
 essentially self-adjoint on $C_0^\infty(\R^d)$; see
 \cite{Dunford,ReedSimon2}. In the linear case $\lambda=0$, classical
 trajectories go to infinity  exponentially fast in time, the solution
 disperses exponentially in time, and the Sobolev norms grow
 exponentially in time (see e.g. \cite{CaSIMA}). Because of that,
 there are no long range effects (scattering theory) when a power-like
 nonlinearity is added (\cite{CaSIMA}), and at least in the case of an
 $L^2$-critical focusing nonlinearity,
 \begin{equation*}
  i\d_t u +\frac{1}{2}\Delta u = -\omega^2\frac{|x|^2}{2}u -
  |u|^{4/d}u,\quad x\in \R^d,
\end{equation*}
there exists no nontrivial  solitary wave $u(t,x)=e^{i\nu t}\phi(x)$
with $\phi\in L^2(\R^d)$ \cite{JoPa93,KaWe94}.
\smallbreak

In the case of \eqref{eq:logNLSrep}, the mass and the energy are
formally independent of time: they are given by
\begin{equation}\label{eq:ME}
\begin{aligned}
   M(u) &= \|u\|_{L^2(\R^d)}^2,\\
   E(u)&=\frac{1}{2}\|\nabla u\|_{L^2(\R^d)}^2-\frac{\omega^2}{2}\|x
   u\|_{L^2}^2+\lambda\int_{\R^d}    |u|^2\(\ln|u|^2-1\)dx.
  \end{aligned}
  \end{equation}
The energy has no definite sign, for two reasons: the repulsive
harmonic potential has a negative contribution in $E$, and the logarithmic
nonlinearity induces a potential energy with indefinite sign
(entropy). Introduce the space $\Sigma$ defined by
\begin{equation*}
  \Sigma =H^1\cap \F(H^1)= \left\{ f\in H^1(\R^d),\quad x\mapsto |x|
    f(x)\in L^2(\R^d)\right\},
\end{equation*}
and  equipped with the norm
\begin{align*}
  \|f\|_{\Sigma}^2 &= \|f\|_{L^2(\R^d)}^2 +\|\nabla f\|_{L^2(\R^d)}^2
                     +\int_{\R^d}|x|^2|f(x)|^2dx \\
  &=  \|f\|_{L^2(\R^d)}^2 +\<\(-\Delta+|x|^2\)f,f\>.
\end{align*}
It is proved in \cite[Proposition~1.3]{CaFe-p} that for $\lambda\in
\R$ and any
$u_0\in \Sigma$, there exists a unique solution $u \in L^\infty_{\rm
  loc}(\R;\Sigma)\cap C(\R;L^2(\R^d))$ to \eqref{eq:logNLSrep}, such
that $u_{\mid t=0}=u_0$. In addition, the mass $M$ and the energy $E$
are independent of time. 
In \cite{CaFe-p}, it is proved in addition that in the case
$\lambda>0$, every solution to \eqref{eq:logNLSrep} disperses
exponentially fast: in particular, there is no solitary wave in this
case.

\smallbreak

The situation is different in the case $\lambda<0$, and leads to
features which appear to be quite unique, in the context of the
logarithmic Schr\"odinger equation (with potential), and more
generally of nonlinear Schr\"odinger equations. In \cite{ZZ20}, it was
proven that \eqref{eq:logNLSrep} admits at least one positive bound
state, under some conditions on the coefficients, recalled below. Under suitable assumptions regarding the
parameters $\lambda$ and $\omega$, we exhibit two positive
stationary solutions.

 \bigbreak

 Due to the presence of the potential, \eqref{eq:logNLSrep} is not
invariant by translation in space, hence the definition below (as in \cite{ACS20}):
\begin{definition}\label{def:stability}
  A standing wave $u(t,x) = \phi(x)
  e^{i\nu t}$ solution to \eqref{eq:logNLSrep} is
  orbitally stable in the energy space if for any $\eps>0$, there exists $\eta>0$ such that if $u_0\in \Sigma$
satisfies $\|u_0-\phi\|_\Sigma<\eta$, then the solution $u$ to
\eqref{eq:logNLSrep} exists for all $t\in \R$, and
 \begin{equation*}
    \sup_{t\in \R}\inf_{\theta\in \R}\|u(t)
    -e^{i\theta}\phi\|_\Sigma<\eps.
  \end{equation*}
  Otherwise, the standing wave is said to be unstable.
\end{definition}
The main result of this paper is the following:
\begin{theorem}\label{theo:main}
  Let $-\lambda>\omega>0$. Then \eqref{eq:logNLSrep} possesses two
  positive stationary solutions, which are Gaussons,
  \begin{equation*}
  \phi_{k_\pm}(x) =
  e^{-\frac{dk_\pm}{4\lambda}}e^{-k_\pm|x|^2/2},\quad\text{where}\quad
k_{\pm}=-\lambda\pm
  \sqrt{\lambda^2-\omega^2}.
\end{equation*}
Each stationary solution generates a continuous family of solitary
waves,
\begin{equation*}
  u_{\pm,\nu}(t,x) =\phi_{k_\pm,\nu}(x)e^{i\nu t},\quad \phi_{k_\pm,\nu}(x) =
  e^{-\frac{\nu}{2\lambda}} \phi_{k_\pm}(x),\quad \nu\in \R.
\end{equation*}
Every such solitary wave is unstable in the sense of
Definition~\ref{def:stability}.\\
In the limiting case $-\lambda=\omega>0$,
$\phi_{k_-}=\phi_{k_+}=\phi_\omega=e^{d/4}e^{-\omega |x|^2/2}$ also  generates a continuous family of solitary
waves,
\begin{equation*}
  u_{\omega,\nu}(t,x) =\phi_{\omega,\nu}(x)e^{i\nu t},\quad \phi_{\omega,\nu}(x) =
  e^{\frac{\nu}{2\omega}} \phi_{\omega}(x),\quad \nu\in \R,
\end{equation*}
and every such solitary wave is unstable in the sense of
Definition~\ref{def:stability}.
\end{theorem}
We note that $\phi_{k_-}$ and $\phi_{k_+}$ are two positive solutions
to the stationary equation
\begin{equation}\label{eq:stationary}
  -\frac{1}{2}\Delta \phi -\omega^2\frac{|x|^2}{2}\phi +\lambda
  \phi\ln\(|\phi|^2\)=0.
\end{equation}
As evoked above, it is
shown in \cite{ZZ20} that \eqref{eq:logNLSrep} has at least one
positive solution, under suitable assumptions on the coefficients of
the equation. More precisely,
in \cite{ZZ20}, a semiclassical parameter $\eps$ is present,
\begin{equation*}
  -\eps^2\Delta u -|x|^2 u = u\ln|u|^2.
\end{equation*}
A stationary, positive solution exists for
sufficiently small values of the semiclassical parameter $\eps$: a
rescaling argument shows that this corresponds to
\eqref{eq:stationary} the case
$\lambda=-2$, with $\omega=\eps$: for $\eps$ small, we indeed have $-\lambda>\omega>0$. In \cite{AlvesJi2020}, it is
shown that for the logarithmic Schr\"odinger equation with a potential
admitting a global minimum reached in $\ell\ge 2$ points sufficiently
far one from another, there exist at
least $\ell$ positive stationary solutions, providing a
situation where nonuniqueness holds, which is quite different from ours.
\smallbreak

Linearizing \eqref{eq:logNLSrep} around $\phi_k$, for $k=k_-$ or $k_+$, leads to:
\begin{equation*}
  i\d_t u+\frac{1}{2}\Delta u = -\omega^2\frac{|x|^2}{2}u
  -\frac{dk}{2}u -\lambda k |x|^2 u = k^2\frac{|x|^2}{2}u-\frac{dk}{2}u .
\end{equation*}
The underlying Hamiltonian is the (shifted) harmonic oscillator,
\begin{equation*}
  H_k = -\frac{1}{2}\Delta + k^2\frac{|x|^2}{2}-\frac{dk}{2},
\end{equation*}
whose point spectrum is $k\N$. This implies \emph{linear and spectral stability} of the
stationary states $\phi_{k_\pm}$, like e.g. for the Gausson in the case of
the logarithmic KdV equation \cite{CaPe14,JaPe14,Pe17}. From this
perspective, the nonlinear instability stated in
Theorem~\ref{theo:main} can appear surprising. We actually show
several possible mechanisms leading to instability.

\smallbreak

Ground states are often characterized as the unique positive
solution to an elliptic equation (typically when the nonlinearity is
homogeneous, but not only, see e.g. \cite{Jang2010,LewinRotaNodari20}): we
discuss more into details the
notion of ground state in Section~\ref{sec:ground}, and show that
neither $\phi_{k_-}$ nor $\phi_{k_+}$ can be considered as a ground
state according to standard definitions. Note that the underlying
operator $-\Delta-\omega^2|x|^2$ is not elliptic, since its symbol is
$|\xi|^2-\omega^2|x|^2$. In particular, we do not obtain a variational
characterization of the Gaussons in the present case, unlike in the
case without potential \cite{Ar16}, or with a confining harmonic
potential \cite{ACS20,CaFe-p}. This is consistent with the fact that
these solutions are unstable. Note however that in view of the global
existence result \cite[Proposition~1.3]{CaFe-p}, the instability
mechanism is not related to finite time blow-up.
\smallbreak

The rest of this paper is organized as follows. In
Section~\ref{sec:special}, we show some special
invariances and  discuss more into details special
Gaussian solutions to \eqref{eq:logNLSrep}. In
Section~\ref{sec:instability}, we complete the proof
of Theorem~\ref{theo:main}, by showing the instability of
$\phi_{k_-}$ and $\phi_{k_+}$; several causes of instability are
exhibited. Finally in Section~\ref{sec:ground}, we discuss the notion
of ground state associated to \eqref{eq:logNLSrep}, and show that it
should be considered that \eqref{eq:logNLSrep} possesses no ground
state.

\section{Special solutions and invariances}
\label{sec:special}

\subsection{Some invariances}

\eqref{eq:logNLSrep} is invariant with respect to translation in time,
but not with respect to translation in space, due to the potential. It
is gauge invariant: if $u$ is a solution, then so is $e^{i\theta} u$
for any constant $\theta\in \R$.
\smallbreak

\subsubsection*{Size effect.} The following invariance is a feature of the logarithmic nonlinearity:
If $u$ solves \eqref{eq:logNLSrep}, then for all
$c\in \C$, so does
\begin{equation}\label{eq:scaling}
u_c(t,x):=  c\, u(t,x) e^{-it\lambda\ln|c|^2}.
\end{equation}
Typically, if we find a stationary solution, then the above transform
generates a continuum of solitary waves, indexed by $c\in (0,\infty)$, or
equivalently by
\begin{equation*}
  \nu = -\lambda\ln \(c^2\)\in \R.
\end{equation*}
Note that the size of these solitary waves is arbitrary, as $c$ ranges
$(0,\infty)$.
\smallbreak

\subsubsection*{Galilean invariance.} Due to the repulsive harmonic potential, the Galilean
invariance reads are follows: If $u(t,x)$ solves
\eqref{eq:logNLSrep}, then  for any $v\in \R^d$, so does
\begin{equation}
  \label{eq:galileo}
  u\(t,x-v\frac{\sinh(\omega t)}{\omega}\)\exp\(i\cosh(\omega t)v\cdot x-
  \frac{i|v|^2}{4\omega}\sinh(2\omega t)\).
\end{equation}
At $t=0$, the above transform is just a multiplication by
$e^{iv\cdot x}$.
\smallbreak

\subsubsection*{Space translation.} The absence of invariance with respect to translation in space can be
specified as follows: If $u$ solves \eqref{eq:logNLSrep}, then for any
$x_0\in \R^d$, so does
\begin{equation}
  \label{eq:translation}
  u\(t,x-x_0\cosh(\omega t)\) \exp\left(i\omega \sinh(\omega t)x_0\cdot x
    -\frac{i\omega|x_0|^2}{4} \sinh(2\omega t)\right).
\end{equation}
At $t=0$, the above transform corresponds to a shift in space.
\smallbreak

\subsubsection*{Tensorization.} The logarithmic nonlinearity was introduced in \cite{BiMy76} to
satisfy the following tensorization property: as the external potential
decouples space variables,
\begin{equation*}
  -\omega^2\frac{|x|^2}{2}= -\frac{\omega^2}{2}\sum_{j=1}^d x_j^2,
\end{equation*}
if the initial datum is a tensor product,
\begin{equation*}
  u_0(x) =\prod_{j=1}^d u_{0j}(x_j),
\end{equation*}
then the solution to \eqref{eq:logNLSrep} is given by
\begin{equation*}
  u(t,x) =\prod_{j=1}^d u_{j}(t,x_j),
\end{equation*}
where each $u_j$ solves a one-dimensional equation,
\begin{equation*}
   i\d_t u_j +\frac{1}{2} \d_{x_j}^2 u_j =-\omega^2\frac{x_j^2}{2}u_j+ \lambda
   \ln\(|u_j|^2\)u_j  ,\quad u_{j\mid t=0} =u_{0j} .
 \end{equation*}

\subsection{Gaussons}

As announced in the introduction, for $-\lambda>\omega>0$, the
stationary Gaussons are  given by
\begin{equation*}
  \phi_k(x) = e^{-\frac{dk}{4\lambda}}e^{-k|x|^2/2},
\end{equation*}
where $k$ is either of the solutions to
\begin{equation}\label{eq:k}
  k^2+2\lambda k +\omega^2=0, \text{ i.e. }k_{\pm}=-\lambda\pm
  \sqrt{\lambda^2-\omega^2}.
\end{equation}
If $-\lambda=\omega>0$, then $k_-=k_+=\omega$, and we will see in the
next subsection that when $\omega>-\lambda>0$, there exists no Gausson.
We compute
\begin{equation*}
  \|\phi_k\|_{L^2(\R^d)}^2 = e^{-\frac{dk}{2\lambda}}\(\frac{\pi}{k}\)^{d/2}.
\end{equation*}
  We note that as $\omega\to 0$ with $\lambda<0$ fixed, $k_-\to 0$, $k_+\to
  -2\lambda$, hence
  \begin{equation*}
    \|\phi_{k_-}\|_{L^2(\R^d)}^2\to \infty,\quad\text{whereas}\quad
    \|\phi_{k_+}\|_{L^2(\R^d)}^2\to e^d \(\frac{\pi}{-2\lambda}\)^{d/2}.
  \end{equation*}
  We have more generally
  \begin{lemma}\label{lem:compare}
    Let $-\lambda>\omega>0$. We have
    \begin{equation*}
      \|\phi_{k_-}\|_{L^2(\R^d)}> \|\phi_{k_+}\|_{L^2(\R^d)}.
    \end{equation*}
  \end{lemma}
  \begin{proof}
    It suffices to prove that
    \begin{align*}
      \frac{e^{-k_-/\lambda}}{k_-}>
      \frac{e^{-k_+/\lambda}}{k_+}&\Longleftrightarrow
      e^{(k_+-k_-)/\lambda}>\frac{k_-}{k_+}\\
      &\Longleftrightarrow
      e^{2\sqrt{\lambda^2-\omega^2}/\lambda}>\frac{-\lambda
        -\sqrt{\lambda^2-\omega^2}}{-\lambda
        +\sqrt{\lambda^2-\omega^2}} .
    \end{align*}
 We view the above inequality as depending on the unknown $\omega\in
(0,-\lambda)$, and change the unknown as $\theta=
\sqrt{\lambda^2-\omega^2}/|\lambda|\in (0,1)$, so the above inequality
becomes
\begin{equation*}
  e^{-2\theta}>\frac{1-\theta}{1+\theta}\Longleftrightarrow
  1+\theta>(1-\theta)e^{2\theta}.
\end{equation*}
The map $f(\theta) = 1+\theta-(1-\theta)e^{2\theta}$, defined for
$\theta\in (0,1)$, satisfies
\begin{equation*}
  f''(\theta) = 4 e^{2\theta}-4(1-\theta)e^{2\theta}>0, \text{ hence }
  f'(\theta) = 1+e^{2\theta} -2(1-\theta)e^{2\theta}> 0,
\end{equation*}
and $f(\theta)>0$ for all $0<\theta<1$.
  \end{proof}
  \bigbreak

In view of \eqref{eq:scaling}, with $\nu = -\lambda \ln (c^2)$, $c>0$, we have a continuum of standing waves:
\begin{equation*}
  u_{\pm,\nu}(t,x) =\phi_{k_\pm,\nu}(x)e^{i\nu t},\quad \phi_{k_\pm,\nu}(x) =
  e^{-\frac{\nu}{2\lambda}} \phi_{k_\pm}(x),\quad \nu\in \R.
\end{equation*}
Therefore, to understand the dynamical properties of $u_{\pm,\nu}$
(orbital stability or instability), it is enough to consider the
stationary solutions $\phi_{k_\pm}$.

\subsection{Gaussian solutions}

By Gaussian solutions, we mean solutions which are Gaussian in the
space variable, with time-dependent coefficients. We adapt the computations presented in \cite{CaGa18} in the case $\omega=0$.
Suppose $d=1$ (for $d\ge 2$, we may invoke the above tensorization
property).
We seek $u(t,x) =
b(t)e^{-a(t)x^2/2}$ (in particular $u_0$ is Gaussian). We
find:
\begin{equation*}
  i\dot b = \frac{1}{2}ab+\lambda b\ln|b|^2;\quad i\dot a
  =a^2+2\lambda\RE a+\omega^2.
\end{equation*}
The function $b$ is given explicitly in terms of $a$ and its initial
value $b_0$,
\begin{equation*}
  b(t) =b_0 \exp\( -i\lambda t\ln\(|b_0|^2\)-\frac{i}{2}A(t)-i\lambda
  \IM\int_0^t A(s)ds\),
\end{equation*}
where we have denoted
$\displaystyle  A(t)=\int_0^t a(s)ds$.
We may write $a$ under the form
\begin{equation}\label{eqa}
  a=\frac{1}{\tau^2} -i\frac{\dot \tau}{\tau},\quad \tau\in\mathbb{R},
\end{equation}
and the equation for $a$ leads to
\begin{equation}\label{eq:tau-general}
  \ddot \tau = \frac{2\lambda}{\tau}+\frac{1}{\tau^3}+\omega^2 \tau.
\end{equation}
We note that the form \eqref{eqa} implies that $b(t)$ can be written
as
\begin{equation}\label{eqb}
  b(t) =b_0 e^{i\theta(t)}\sqrt{\frac{\tau(0)}{\tau(t)}},\quad
  \theta(t)\in \R.
\end{equation}
Multiplying \eqref{eq:tau-general} by $\dot{\tau}$ and integrating, we
get
\begin{equation}\label{eq:tau-energy}
  (\dot \tau)^2 =C_0+ 4\lambda \ln|\tau| -\frac{1}{\tau^2}+\omega^2 \tau^2,
\end{equation}
where $C_0=\dot{\tau}(0)^2-4\lambda \ln|\tau(0)|+\frac{1}{\tau(0)^2}-\omega^2\tau(0)^2$ is related to the  initial data. Noticing that $F(q)=C_0+ 4\lambda \ln q -\frac{1}{q^2}+\omega^2 q^2\to -\infty$ when $q\to 0$, this readily shows that $\tau$
remains bounded away from zero, and thus may be supposed positive in
view of \eqref{eqa}:
\begin{equation*}
  \exists \delta>0,\quad \tau(t)\ge\delta,\quad\forall t\ge 0.
\end{equation*}
\begin{proposition}
  Let $d=1$, $\lambda<0<\omega$. \\
 $1.$ If $-\lambda>\omega>0$, then \eqref{eq:tau-general} has exactly two
 stationary solutions,
$\tau_\mp=1/\sqrt{k_\pm}$.
The other solutions are either periodic, or unbounded, corresponding to
time-periodic and dispersive Gaussian solutions to
\eqref{eq:logNLSrep}, respectively.\\
$2.$ If $-\lambda=\omega>0$, then \eqref{eq:tau-general} has exactly
one stationary solution, $\tau_0= 1/\sqrt\omega$. All the
  other solutions are unbounded. In other words, any Gaussian solution
 to  \eqref{eq:logNLSrep}  which is not of the form
  \begin{equation*}
    e^{(2\nu+\omega)/(4\omega)}e^{i\nu t} e^{-\omega x^2/2},\quad \nu
    \in \R,
  \end{equation*}
  is dispersive.\\
  $3.$ If $\omega>-\lambda>0$, then every solution to
  \eqref{eq:tau-general} is unbounded. More precisely,
  \begin{equation*}
     e^{\omega t}\lesssim \tau(t)\lesssim
   e^{\omega t}, \quad t \ge 0,
  \end{equation*}
  and every Gaussian solution to
  \eqref{eq:logNLSrep} disperses exponentially fast.
\end{proposition}
\begin{proof}
  We remark that the righthand side of \eqref{eq:tau-general} can be
  rewritten as
  \begin{equation*}
    \ddot \tau=P\(\frac{1}{\tau^2}\)\tau,\quad P(X)= X^2+2\lambda X+\omega^2.
  \end{equation*}
 When $-\lambda>\omega>0$, $P$ has exactly two roots, $k_-$ and
 $k_+$, so
 \begin{equation*}
   \ddot \tau = \(\frac{1}{\tau^2}-k_-\)\(\frac{1}{\tau^2}-k_+\)\tau.
 \end{equation*}
According to the initial data for $\tau$, the value of the constant
$C_0$ in \eqref{eq:tau-energy} varies,  leading to bounded
trajectories, in which case $\tau$ is periodic, or to unbounded
trajectories, in which case $\tau(t)\to \infty$ as $t$ goes to
infinity. This is illustrated by Figure~\ref{phase}, displaying the
phase portrait for the equation \eqref{eq:tau-general} with $\omega=1$
and $\lambda=-2$, where we find
\begin{equation*}
  \tau_-=\frac{1}{\sqrt{2+\sqrt{3}}}\approx 0.518,\quad
  \tau_+=\frac{1}{\sqrt{2-\sqrt{3}}}\approx  1.932.
\end{equation*}

\begin{figure}[htbp]
\begin{center}
\includegraphics[width=4in,height=2.5in]{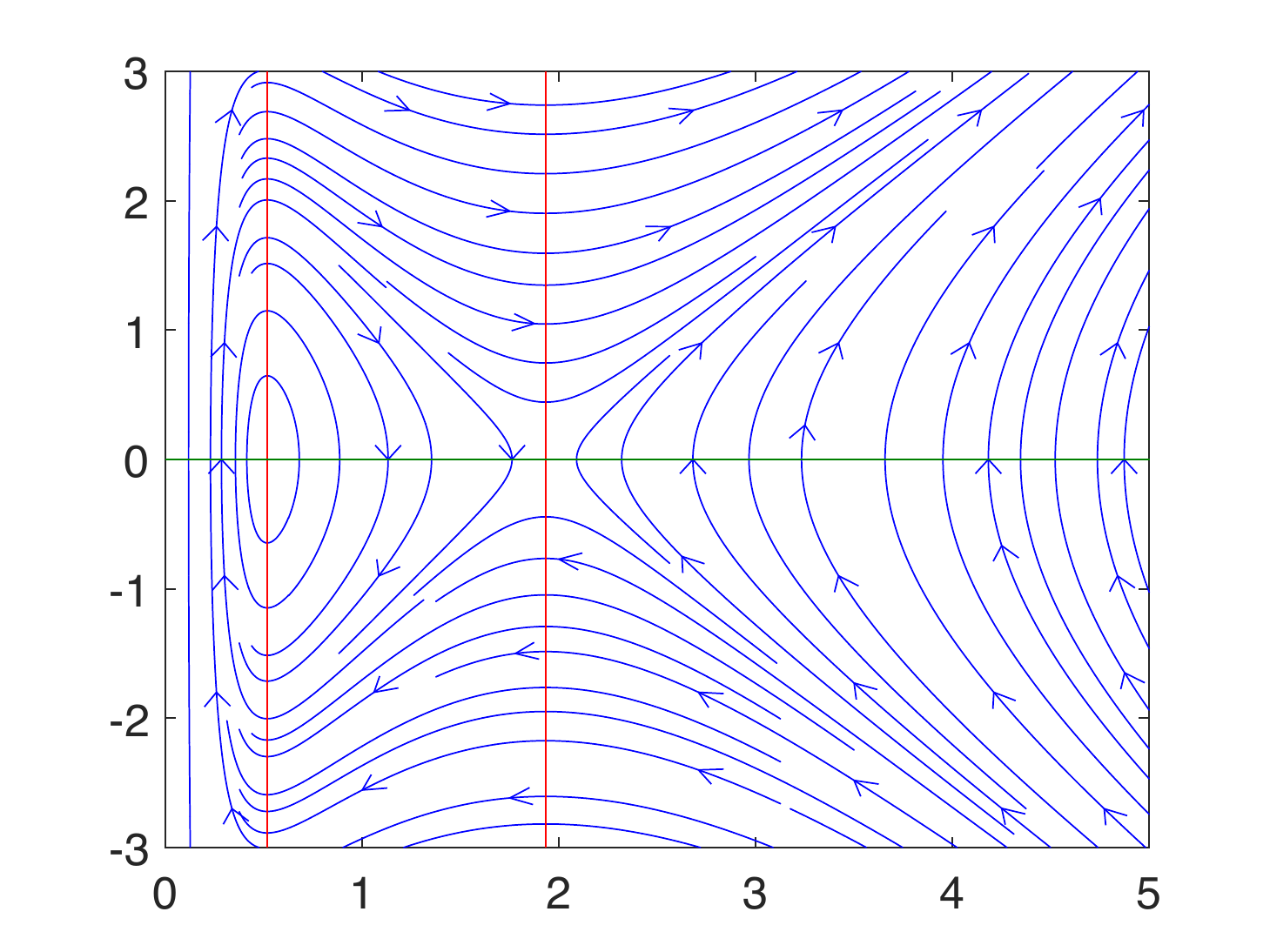}
\end{center}
\caption{Phase portraits for the ODE \eqref{eq:tau-general} with $\omega=1$ and $\lambda=-2$.}
\label{phase}
\end{figure}
When $-\lambda=\omega>0$, $P$ has exactly one double root $\omega$, and
\begin{equation*}
   \ddot \tau = \(\frac{1}{\tau^2}-\omega\)^2\tau.
 \end{equation*}
If $\tau$ is not constant (equal to $1/\sqrt\omega$), then $\tau$ is
strictly convex. If $\tau(t_0)=1/\sqrt\omega$ for some $t_0\ge 0$,
then $\dot \tau(t_0)\not =0$, for otherwise $\tau$ would be constant,
by uniqueness for \eqref{eq:tau-general}: $\tau$ can't remain close
to $1/\sqrt\omega$, and assuming that $\tau$ is bounded leads to a
contradiction. As $\tau$ is positive and convex, $\tau(t)$ goes to
infinity as $t\to \infty$. This is illustrated in Figure~\ref{phase2}.
\smallbreak

\begin{figure}[htbp]
\begin{center}
\includegraphics[width=4in,height=2.5in]{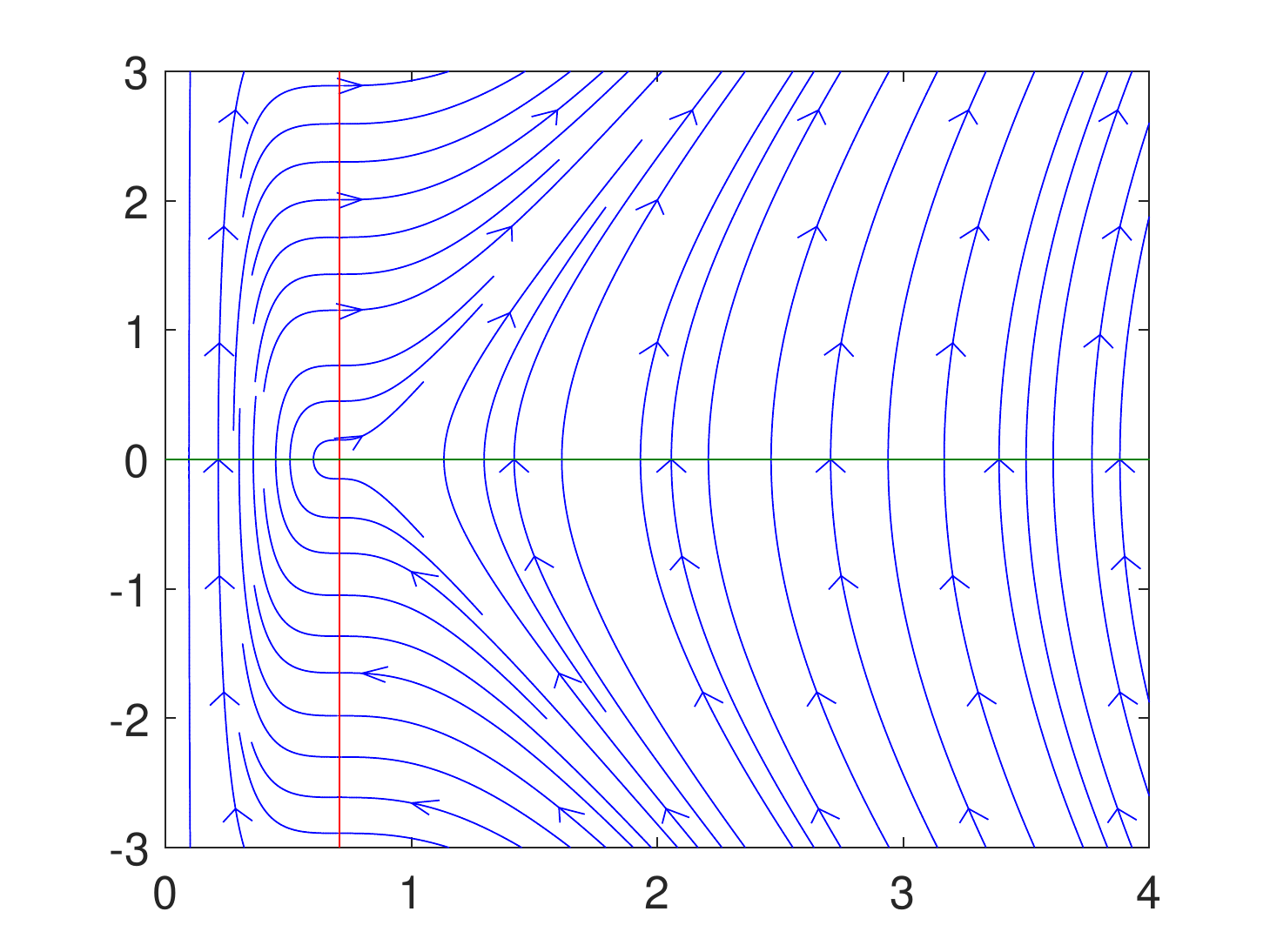}
\end{center}
\caption{Phase portraits for the ODE \eqref{eq:tau-general} with $\omega=2$ and $\lambda=-2$.}
\label{phase2}
\end{figure}

When $ \omega>-\lambda>0$, $P$  is uniformly bounded from below on
$\R$, $P(X)\ge \delta>0$. If $\tau$ was bounded, \eqref{eq:tau-general}
would yield $\ddot \tau\gtrsim 1$, since $\tau$ is bounded away from zero, hence a contradiction. As $\tau$ is
convex, $\tau(t)$ goes to
infinity as $t\to \infty$, see Figure~\ref{phase3}.
\begin{figure}[htbp]
\begin{center}
\includegraphics[width=4in,height=2.5in]{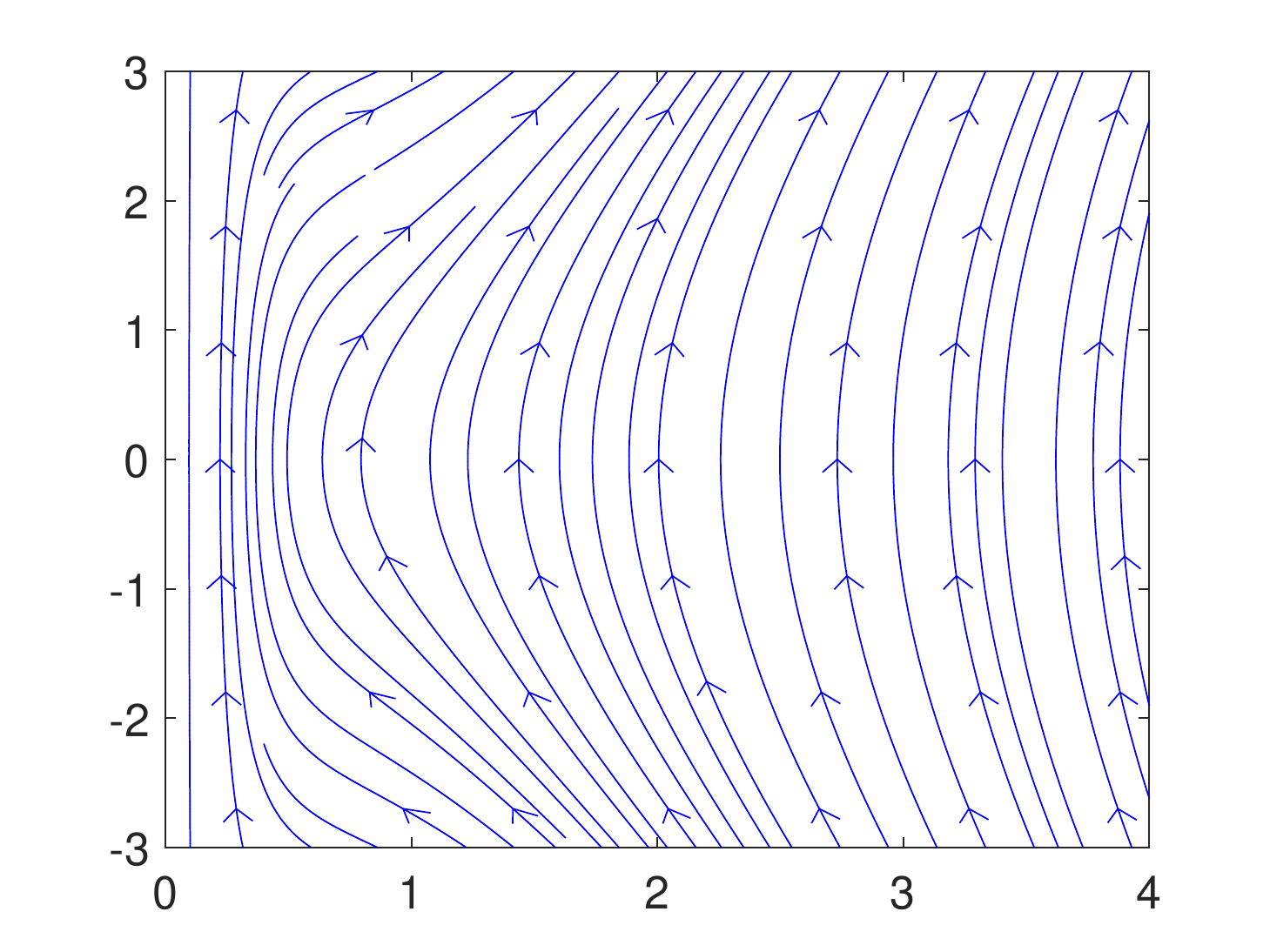}
\end{center}
\caption{Phase portraits for the ODE \eqref{eq:tau-general} with $\omega=2$ and $\lambda=-1$.}
\label{phase3}
\end{figure}
As a consequence,
for any $\eps>0$, picking $T$ sufficiently large,
\begin{equation*}
\ddot \tau(t) \ge \omega^2\tau(t)-\eps, \quad
\forall t\ge T.
\end{equation*}
The solution to
\begin{equation*}
 \ddot \theta(t)= \omega^2\theta(t)-\eps, \quad
 \theta(T)=\tau(T),\quad \dot\theta(T)=\dot\tau(T),
\end{equation*}
is given by
\begin{equation*}
 \theta(t) =\tau(T)\cosh\(\omega(t-T)\) +
 \dot\tau(T)\frac{\sinh\(\omega(t-T)\) }{\omega} -\frac{2\eps}{\omega^2}
            \sinh^2\(\frac{\omega}{2}(t-T)\).
\end{equation*}
As $\tau(T)$ and $\dot\tau(T)$ go to infinity as $T\to\infty$, we
infer that $\tau(t) \gtrsim e^{\omega t}$. The converse
estimate is a direct consequence of \eqref{eq:tau-energy}, again
because for $t$ sufficiently large, $\ln\tau(t)>0$, and $\lambda<0$.
\end{proof}

\section{Orbital instability}
\label{sec:instability}

The instability result that we prove is slightly stronger than
instability in the sense of Definition~\ref{def:stability}:
\begin{lemma}\label{lem:unstable}
  Let $\nu\in \R$.\\
$1.$ Suppose $-\lambda>\omega>0$.
  The solitary waves $\phi_{k_-,\nu}(x)e^{i\nu t}$
  and $\phi_{k_+,\nu}(x)e^{i\nu t}$
  are unstable. More precisely, for any $\eta>0$, there exists $u_0\in
  \Sigma$ such that
  \begin{equation*}
    \|u_0-\phi_{k_+,\nu}\|_{\Sigma}<\eta,
  \end{equation*}
 and the solution to \eqref{eq:logNLSrep} such that $u_{\mid t=0}=u_0$
 satisfies
 \begin{equation*}
   \sup_{t\ge 0} \inf_{\theta\in \R}\left\| u(t) -
     e^{i\theta}\phi_{k_+,\nu}\right\|_{L^2(\R^d)} \ge
   \frac{1}{2}\|\phi_{k_+,\nu}\|_{L^2(\R^d)} .
 \end{equation*}
 The same holds when $k_+$ is replaced by $k_-$.\\
 $2.$ Suppose $-\lambda=\omega>0$. The solitary wave
 $\phi_{\omega,\nu}(x)e^{i\nu t}$ is unstable in the same sense as
 above.
\end{lemma}
\begin{proof}
We present the argument for $\phi_{k_+}$, to shorten notations: considering
$\phi_{k_\pm,\nu}$ for $\nu\in \R$ goes along the same lines, and the
argument includes the limiting case $-\lambda=\omega>0$. For all $\eta>0$,
then exists $\delta>0$ such that for $|x_0|<\delta$,
\begin{equation*}
  \|u_0-\phi_{k_+}\|_\Sigma<\eta,\quad u_0(x) = \phi_{k_+}(x-x_0).
\end{equation*}
In view of \eqref{eq:translation}, the solution to \eqref{eq:logNLSrep} with
initial datum $u_0$ is given by
\begin{equation*}
  u(t,x) = \phi_{k_+}\(x-x_0\cosh(\omega t)\)e^{i\omega \sinh(\omega
    t)x_0\cdot x -\frac{i\omega|x_0|^2}{4}\sinh(2\omega
    t)}.
\end{equation*}
Therefore, for any $t>0$,
\[\inf_{\theta\in \R}\|u(t)
    -e^{i\theta}\phi_{k_+}\|^2_{L^2(\R^d)} \ge \int_{\R^d}\left|
      \phi_{k_+}\(x-x_0\cosh(\omega t)\) - \phi_{k_+}(x)\right|^2dx.\]
Indeed, denote $u(t,x)=\phi_{k_+}\(x-x_0\cosh(\omega
t)\)e^{i\alpha(x_0, x, t)}$ with $\alpha (x_0, x, t)\in \R$ given by
the above formula. Then
\begin{align*}
&\|u(t)-e^{i\theta}\phi_{k_+}\|^2_{L^2(\R^d)}=\|\phi_{k_+}\(x-x_0\cosh(\omega t)\)-e^{i(\theta-\alpha(x_0, x,t))}\phi_{k_+}\|^2_{L^2(\R^d)}\\
&\quad=2\|\phi_{k_+}\|^2_{L^2(\R^d)}-2\int_{\R^d}
\cos(\theta-\alpha(x_0, x,t))\phi_{k_+}\(x-x_0\cosh(\omega t)\)\phi_{k_+}(x)dx,
\end{align*}
which implies
\begin{align*}
&\quad \inf_{\theta\in \R}\|u(t)
    -e^{i\theta}\phi_{k_+}\|^2_{L^2(\R^d)}\\
   & =2\|\phi_{k_+}\|^2_{L^2(\R^d)}-2
\sup_{\theta\in \R}\int_{\R^d}
\cos(\theta-\alpha(x_0, x,t))\phi_{k_+}\(x-x_0\cosh(\omega t)\)\phi_{k_+}(x)dx\\
&\ge 2\|\phi_{k_+}\|^2_{L^2(\R^d)}-2
\int_{\R^d}\phi_{k_+}\(x-x_0\cosh(\omega t)\)\phi_{k_+}(x)dx\\
&=\left\|
      \phi_{k_+}\(x-x_0\cosh(\omega t)\) - \phi_{k_+}(x)\right\|_{L^2(\R^d)}^2.
\end{align*}
It becomes obvious that picking $t$ sufficiently large (in terms
  of $\eta$) leads to
  \[\inf_{\theta\in \R}\|u(t)
    -e^{i\theta}\phi_{k_+}\|^2_{L^2(\R^d)} \ge
    \frac{1}{2}\|\phi_{k_+}\|_{L^2(\R^d)}^2.\]
  This rules out orbital stability, even in the $L^2$-norm, for
  initial data close to $\phi_{k_+}$ in the $\Sigma$-topology.
\end{proof}
\begin{remark}
   We can adapt the above proof by using the Galilean
   invariance \eqref{eq:galileo}, and consider instead
   \[ u_0(x) = \phi_{k_+}(x)e^{iv\cdot x},\quad |v|\ll 1.\]
\end{remark}
\begin{remark}
  It is clear from the argument that $u_0$ is close to $\phi_{k_+}$ in
  $\Sigma$, but also in stronger norms, while orbital stability is
  ruled out by measuring only the $L^2$-norm.
\end{remark}
The above arguments do not rule out orbital stability when the initial datum are  restricted to be radially symmetric. In \cite{Caz83}, this
    restriction was considered essentially to obtain compactness
    properties (the embedding of $H^1_{\rm rad}(\R^d)$ into
    $L^p(\R^d)$ for $2\le p<\frac{2}{(d-2)_+}$ is compact). Note that
      $\Sigma$ is compactly embedded into  $L^p(\R^d)$ for $2\le
      p<\frac{2}{(d-2)_+}$.
    The lemma below shows
    instability for $\phi_{k_-}$
    even at the radial level.

\begin{lemma}\label{lem:unstable-lin}
  Let $\nu\in \R$. \\
  $1.$ Suppose $-\lambda>\omega>0$. The solitary wave $\phi_{k_-,\nu}(x)e^{i\nu t}$
  is unstable even if we restrict Definition~\ref{def:stability} to
  radial solutions.\\
  $2.$ The same holds for $\phi_{\omega,\nu}(x)e^{i\nu t}$ in the case
  $-\lambda=\omega>0$.
\end{lemma}
\begin{proof}
  Assume $-\lambda>\omega>0$. We show that $u_{k_-,\nu}$ is unstable even as a Gaussian solution centered at
  the origin, by linearizing
  \eqref{eq:tau-general} about $\tau_-=1/\sqrt{k_-}$: we compute the linearization
  as
  \begin{equation*}
    \ddot h = \omega^2 h -2\lambda k_- h -3k_-^2 h = \Omega_{\rm eff}h,
  \end{equation*}
  where
  \[\Omega_{\rm eff}=\omega^2 -2\lambda k_- - 3k_-^2= -4k_-^2 -4\lambda
    k_- =-4k_-(k_-+\lambda). \]
  Since $k_-+\lambda<0$, the linearized operator is such that
  $\Omega_{\rm eff}>0$, so $h$ grows exponentially. Of course
  linearizing makes sense only for sufficiently small $h$, but this
  is enough to contradict the definition of orbital stability. Indeed,
  there exists $\delta>0$ such that as long as $|h(t)|\le \delta$, we
  can write the solution $\tau$ to \eqref{eq:tau-general} with
  $\tau(0)= \tau_-+h(0)$ and $\dot\tau(0)=0$ as
  \begin{equation*}
    \tau(t) = \tau_-+h(t)+r(t),\quad \text{with}\quad |r(t)|\le
    \frac{|h(t)|}{2}.
  \end{equation*}
  For $0<\eps<\delta$, let $h$ solve
  \begin{equation*}
    \ddot h = \Omega_{\rm eff}h,\quad h(0)=\eps,\quad \dot h(0)=0.
  \end{equation*}
  As $h(t) = \eps\cosh(t\sqrt{\Omega_{\rm eff}} )$ grows exponentially, there exists $t_0>0$ such that
  $h(t_0)=\delta$, and the triangle inequality yields
  \begin{equation*}
    |\tau(t_0)-\tau_-|\ge \frac{\delta}{2}.
  \end{equation*}
  Now if $u$ denotes the Gaussian solution associated with $\tau$, we
  see that for all $\eta>0$, picking $\eps>0$ sufficiently small
  ensures
  \begin{equation*}
    \|u(0)-\phi_{k_-}\|_\Sigma<\eta,
  \end{equation*}
  while, in view of \eqref{eqb}, setting $k(t) = 1/\tau(t)^2$,
  \begin{align*}
    \sup_{t\ge 0}\inf_{\theta\in
      \R}\|u(t)-e^{i\theta}&\phi_{k_-}\|_{L^2(\R^d)}\ge \inf_{\theta\in
     \R}\|u(t_0)-e^{i\theta}\phi_{k_-}\|_{L^2(\R^d)}\\
&\ge e^{-dk_-/(4\lambda)} \left\|\(\frac{\tau_-}{\tau(t_0)}\)^{d/2}
 e^{-k(t_0)|x|^2/2}- e^{-k_-|x|^2/2}\right\|_{L^2(\R^d)}\\
 &\ge C(\delta)>0,
  \end{align*}
  where $C(\delta)$ is independent of $\eps$, hence independent of $\eta$.
  Thus, we have the same instability results as in Lemma~\ref{lem:unstable}, at
  the level of radial Gaussian solutions.
  \smallbreak

  In the case $-\lambda=\omega>0$, we find $\Omega_{\rm eff}=0$, hence
  $h(t) = \dot h(0)t+h(0)$.
  We now pick $\dot h(0)=\eps$, $h(0)=0$, so $h$ is still unbounded as
  time grows.
  We  thus consider the solution $\tau$ to \eqref{eq:tau-general} with
  $\tau(0)= \tau_-\(=1/\sqrt\omega\)$ and $\dot\tau(0)=\eps$, and the
  above argument can be repeated.
\end{proof}
\begin{remark}
  For $-\lambda>\omega>0$, the same argument is not conclusive in the case of $k_+$, since we   then have
  \begin{equation*}
    \Omega_{\rm eff} =-4k_+(k_++\lambda)<0.
  \end{equation*}
  The trajectories of the linearized operator are bounded
  (periodic). This is consistent with the phase portrait corresponding
  to the Gaussian case, see Figure~\ref{phase} (recalling that $k_+$
  corresponds to the smaller value $\tau_-$). 
\end{remark}

\section{On the notion of ground state}
\label{sec:ground}

The most standard notions of ground state seem to be the following:
\begin{itemize}
\item Minimizer of the action $E+\nu M$.
\item Minimizer of the energy $E$ for a given mass $M$.
\item Positive solution of $dE+\nu dM=0$.
\end{itemize}
In the case of an homogeneous nonlinearity, the three notions
coincide, and the ground state is unique, up to the invariants of the
equation; see e.g. \cite[Chapter~8]{CazCourant}. In the absence of
potential ($\omega=0$), the Gausson is the only positive stationary
solution to \eqref{eq:logNLSrep} \cite{Troy2016}.
In the present case, we have seen already that for
$\-\lambda>\omega>0$, there are two
distinct solutions to the stationary equation $dE=0$, namely
$\phi_{k_-}$ and $\phi_{k_+}$: the last notion cannot be relevant. On
the other hand, because the potential is unbounded from below, the
first two notions are not relevant either: given $u\in \Sigma$,
\begin{equation*}
  E(u_{x_0})\Tend {|x_0|} \infty -\infty,\quad u_{x_0}(x):= u(x-x_0).
\end{equation*}
In \cite{BeJe16}, the second notion is adapted, by requiring in
addition that the ground state is a critical point of the energy on
the set of function with a given mass $M$, which is meaningful even
when the energy is unbounded from below on this set. The case of the
logarithmic nonlinearity turns out to be rather specific: a solitary
wave $e^{i\nu t}\phi(x)$ solves \eqref{eq:logNLSrep} if and only if
$\phi$ solves
\begin{equation*}
  -\frac{1}{2}\Delta \phi +\nu \phi-\omega^2\frac{|x|^2}{2}\phi
  +\lambda \phi\ln|\phi|^2=0.
\end{equation*}
Multiplying this equation by $\bar \phi$ and integrating shows that
$\phi$ must solves
\begin{equation*}
  \|\nabla \phi\|_{L^2}^2 - \omega^2 \|x
                \phi\|_{L^2}^2 + 2 \lambda \int_{\R^d} |\phi|^2
                \ln{|\phi|^2} d x + 2 \nu \|\phi\|_{L^2}^2=0.
\end{equation*}
This Pohozaev identity defines the \emph{Nehari manifold}. But we see
that the above left hand side differs from twice the energy
\begin{equation*}
   E(u)=\frac{1}{2}\|\nabla u\|_{L^2(\R^d)}^2-\frac{\omega^2}{2}\|x u\|_{L^2}^2+\lambda\int_{\R^d}
    |u|^2\(\ln|u|^2-1\)dx
\end{equation*}
only by the term $2(\lambda+\nu) M$.
Following \cite{Ar16,ACS20} (see also \cite{SqSz2015,Shuai2019}), we
thus introduce
 the action and the Nehari functional,
\begin{align*}
    S_\nu (u) &:= E(u) + \nu \|u\|_{L^2}^2, \\
    I_\nu (u) &:=\|\nabla u\|_{L^2}^2 - \omega^2 \|x
                u\|_{L^2}^2 + 2 \lambda \int_{\R^d} |u|^2
                \ln{|u|^2} d x + 2 \nu \|u\|_{L^2}^2=
                2S_\nu(u)+2\lambda\|u\|_{L^2}^2 ,
\end{align*}
and consider the minimization problem
\begin{align*}
    \delta(\nu) :=& \inf \{ S_{\nu} (u) \, | \, u \in \Sigma
              \setminus \{ 0 \}, I_{\nu} (u) = 0 \} \\
        =& - \lambda \inf \{ \|u\|_{L^2}^2 \, | \, u \in \Sigma
       \setminus \{ 0 \}, I_{\nu} (u) = 0 \}.
\end{align*}
The set of ground states is defined by
\begin{equation*}
    \mathcal{G}_\nu :=\{ \phi \in \Sigma \setminus \{ 0 \} \, | \, I_{\nu} (u) = 0, S_\nu (\phi) = \delta(\nu) \}.
\end{equation*}
We check that
\[ I_0(\phi_{k_\pm})=0\quad (\text{hence }I_\nu(\phi_{k_\pm,\nu})=0).\]
In view of Lemma~\ref{lem:compare}, $\phi_{k_-}$ does not belong to
$\mathcal{G}_0$, and should thus not be considered as a ground state,
even though it is a positive solution to \eqref{eq:stationary}.
\bigbreak

It turns out that $\phi_{k_+}$ is not a ground state either:
\begin{proposition}
Let $\lambda<0<\omega$.   For any $\nu\in \R$, $\delta(\nu)=0$, and $ \mathcal{G}_\nu =\emptyset$.
\end{proposition}
\begin{proof}
  Consider the two-parameter family of Gaussians
  \begin{equation*}
    \gamma_{\eps,x_0}(x)=\eps \, e^{-|x-x_0|^2/2}.
  \end{equation*}
  Naturally, the parameter $\eps>0$ is aimed at being arbitrarily
  small, and we use the center $x_0$ to adjust the size of the
  momentum so that $\gamma_{\eps,x_0}$ belongs to the Nehari
  manifold.
  The choice of a variance equal to one is arbitrary, for the
  following computation would lead to the same conclusion for any fixed
  variance. We compute:
  \begin{equation*}
    \|\gamma_{\eps,x_0}\|_{L^2(\R^d)}^2=\eps^2\pi^{d/2},\quad \|\nabla
    \gamma_{\eps,x_0}\|_{L^2(\R^d)}^2=\eps^2\frac{d}{2}\pi^{d/2} ,
  \end{equation*}
  \begin{equation*}
    \|x\gamma_{\eps,x_0}\|_{L^2(\R^d)}^2= \eps^2 \int_{\R^d}|y+x_0|^2
    e^{-|y|^2}dy = \eps^2\frac{d}{2}\pi^{d/2}+ \eps^2|x_0|^2 \pi^{d/2},
  \end{equation*}
  \begin{equation*}
    \int_{\R^d}\gamma_{\eps,x_0}^2\ln\(\gamma_{\eps,x_0}^2\) =
    \ln(\eps^2) \|\gamma_{\eps,x_0}\|_{L^2(\R^d)}^2-\|\nabla
    \gamma_{\eps,x_0}\|_{L^2(\R^d)}^2 = \eps^2 \pi^{d/2}\( \ln(\eps^2) -\frac{d}{2}\),
  \end{equation*}
  hence:
  \begin{equation*}
    I\( \gamma_{\eps,x_0}\) = \eps^2 \pi^{d/2}\(
    (1-2\lambda)\frac{d}{2}
    -\omega^2\frac{d}{2}-\omega^2|x_0|^2+2\lambda \ln(\eps^2)+2\nu \).
  \end{equation*}
  For $\eps>0$ sufficiently small, $2\lambda \ln(\eps^2) +  (1-2\lambda)\frac{d}{2}
    -\omega^2\frac{d}{2}+2\nu>0$ (recall that $\lambda<0$),  and we can find
    $x_0\in \R^d$ (with $|x_0|$ of order $\sqrt{-\ln \eps}/\omega$) such that
    $ I\( \gamma_{\eps,x_0}\)=0$. But of course
    $\|\gamma_{\eps,x_0}\|_{L^2(\R^d)}$ is arbitrarily small, hence
    $\delta(\nu)=0$. The second line in the definition of
    $\delta(\nu)$ obviously implies that $\mathcal{G}_\nu =\emptyset$.
\end{proof}

\begin{remark}
  In the linear case $\lambda=0<\omega$, there is no ground state, and
  more generally, there is no solitary wave, as every solution is
  dispersive. This can be seen for instance \emph{via} the vector field $J(t) = \omega x\sinh(\omega t)+i\cosh(\omega t )\nabla$: as
  observed in \cite[Lemma~2.3]{CaSIMA}, if $u$ solves
  \begin{equation*}
    i\d_t u+\frac{1}{2}\Delta u =-\omega^2\frac{|x|^2}{2}u,
  \end{equation*}
  then so does $Ju$, and since $J$ can be factorized as
  \begin{equation*}
    J(t) = i\cosh(\omega t) e^{i\omega \frac{|x|^2}{2}\tanh(\omega t)}
    \nabla\( e^{-i\omega \frac{|x|^2}{2}\tanh(\omega t)}\cdot\),
  \end{equation*}
  Gagliardo-Nirenberg inequality yields, for $2\le
  p<\frac{2}{(d-2)_+}$,
\begin{align*}
  \|u(t)\|_{L^p(\R^d)}&\le \frac{C(p,d)}{(\cosh(\omega t ))^{\delta(p)}}
  \|u(t)\|_{L^2}^{1-\delta(p)}\|J(t)u\|_{L^2}^{\delta(p) }\\
  &=
    \frac{C(p,d)}{(\cosh(\omega t ))^{\delta(p)} }
    \|u_0\|_{L^2}^{1-\delta(p)}\|\nabla u_0\|_{L^2}^{\delta(p) },\quad \delta(p)=d\(\frac{1}{2}-\frac{1}{p}\),
\end{align*}
since the $L^2$-norm is preserved by the flow. Therefore, if $u_0\in
\Sigma$, the $L^p$-norm of $u$ decreases exponentially in time, and no
solitary wave exists. The existence of solitary waves when $-\lambda\ge
\omega>0$ is thus due to the presence of the logarithmic nonlinearity,
which is sufficiently strong (due to the singularity of the logarithm
at the origin) to counterbalance the exponential linear dispersion.
\end{remark}

\bibliographystyle{abbrv}
\bibliography{biblio}
\end{document}